\documentclass{amsart}
\usepackage{amsmath}
\usepackage{amssymb}
\usepackage{amsthm}
\usepackage{mathrsfs}

\usepackage{graphicx}

\newtheorem{theorem}{Theorem}[section]
\newtheorem{lemma}[theorem]{Lemma}

\theoremstyle{definition}
\newtheorem{definition}[theorem]{Definition}
\newtheorem{example}[theorem]{Example}
\newtheorem{corollary}[theorem]{Corollary}

\newtheorem {obs} [theorem]{Observation}
\theoremstyle{remark}
\newtheorem{remark}[theorem]{Remark}

\numberwithin{equation}{section}

\begin{document}

\title{Ideals in $B_1(X)$ and residue class rings of $B_1(X)$ modulo an ideal}
\author{A. Deb Ray}
\address{Department of Pure Mathematics, University of Calcutta, 35, Ballygunge Circular Road, Kolkata - 700019, INDIA} 
\email{debrayatasi@gmail.com}
\author{Atanu Mondal}
\thanks{The second author is supported by Council of Scientific and Industrial Research, HRDG, India. Sanction No.- 09/028(0998)/2017-EMR-1}

\address{Department of Pure Mathematics, University of Calcutta, 35, Ballygunge Circular Road, Kolkata - 700019, INDIA}
\email{mail.atanu12@yahoo.com}

\begin{abstract}
This paper explores the duality between ideals of the ring $B_1(X)$ of all real valued Baire one functions on a topological space $X$ and typical families of zero sets, called $Z_B$-filters, on $X$. As a natural outcome of this study, it is observed that $B_1(X)$ is a Gelfand ring but non-Noetherian in general. Introducing fixed and free maximal ideals in the context of $B_1(X)$, complete descriptions of the fixed maximal ideals of both $B_1(X)$ and $B_1^*(X)$ are obtained. Though free maximal ideals of $B_1(X)$ and those of $B_1^*(X)$ do not show any relationship in general, their counterparts, i.e., the fixed maximal ideals obey natural relations. It is proved here that for a perfectly normal $T_1$ space $X$, free maximal ideals of $B_1(X)$ are determined by a typical class of Baire one functions. In the concluding part of this paper, we study residue class ring of $B_1(X)$ modulo an ideal, with special emphasize on real and hyper real maximal ideals of $B_1(X)$.
\end{abstract}


\maketitle
\section{Introduction}
\noindent In \cite{AA}, we have introduced the ring of Baire one functions defined on any topological space $X$ and have denoted it by $B_1(X)$. It has been observed that $B_1(X)$ is a commutative lattice ordered ring with unity containing the ring $C(X)$ of continuous functions as a subring. The collection of bounded Baire one functions, denoted by $B_1^*(X)$, is a commutative subring and sublattice of $B_1(X)$. Certainly, $B_1^*(X) \cap C(X) = C^*(X)$.\\\\
\noindent In this paper, we study the ideals, in particular, the maximal ideals of $B_1(X)$ (and also of $B_1^*(X)$). There is a nice interplay between the ideals of $B_1(X)$ and a typical family of zero sets (which we call a $Z_B$-filter) of the underlying space $X$. As a natural consequence of this duality of 
ideals of $B_1(X)$ and $Z_B$-filters on $X$, we obtain that $B_1(X)$ is Gelfand and in general, $B_1(X)$ is non-Noetherian.\\\\
\noindent Introducing the idea of fixed and free ideals in our context, we have characterized the fixed maximal ideals of $B_1(X)$ and also those of $B_1^*(X)$. We have shown that although fixed maximal ideals of the rings $B_1(X)$ and $B_1^*(X)$ obey a natural relationship, the free maximal ideals fail to do so. However, for a perfectly normal $T_1$ space $X$, free maximal ideals of $B_1(X)$ are determined by a typical class of Baire one functions.\\\\
\noindent In the last section of this paper, we discuss residue class ring of $B_1(X)$ modulo an ideal and introduce real and hyper-real maximal ideals in $B_1(X)$.

\section{$Z_B$-filters on $X$ and Ideals in $B_1(X)$}
\begin{definition}
	A nonempty subcollection $\mathscr{F}$ of $Z(B_1(X))$ is said to be a $Z_B$-filter on $X$, if it satisfies the following conditions:
	\begin{enumerate} 
		\item $\emptyset \notin \mathscr{F}$
		\item if $Z_1, Z_2 \in \mathscr{F}$, then $Z_1 \cap Z_2 \in \mathscr{F}$
		\item if $Z \in \mathscr F$ and $Z' \in Z(B_1(X))$ is such that $Z \subseteq Z'$, then $Z' \in \mathscr F$.
	\end{enumerate}
\end{definition}
\noindent Clearly, a $Z_B$-filter $\mathscr F$ on $X$ has finite intersection property. Conversely, if a subcollection $\mathscr B \subseteq Z(B_1(X))$ possesses finite intersection property, then $\mathscr B$ can be extended to a $Z_B$-filter $\mathscr F(\mathscr B)$ on $X$, given by $\mathscr F(\mathscr B)= \{Z \in Z(B_1(X))$: there exists a finite subfamily $\{B_1,B_2,...,B_n\}$ of $\mathscr B$ with $Z \supseteq \bigcap\limits_{i=1}^n B_i$\}. Indeed this is the smallest $Z_B$-filter on $X$ containing $\mathscr B$.
\begin{definition}
	A $Z_B$-filter $\mathscr U$ on $X$ is called a $Z_B$-ultrafilter on $X$, if there does not exist any $Z_B$-filter $\mathscr F$ on $X$, such that $\mathscr U \subsetneqq \mathscr F$.
\end{definition}
\begin{example}
	Let, $A_0=\{Z \in Z(B_1(\mathbb{R})): 0 \in Z\}$. Then $A_0$ is a $Z_B$-ultrafilter on $\mathbb{R}$.
\end{example}
\noindent Applying Zorn's lemma one can show that, every $Z_B$-filter on $X$ can be extended to a $Z_B$-ultrafilter.
Therefore, a family $\mathscr B$ of $Z(B_1(X))$ with finite intersection property can be extended to a $Z_B$-ultrafilter on $X$.
\begin{remark}
	A $Z_B$-ultrafilter $\mathscr U$ on $X$ is a subfamily of $Z(B_1(X))$ which is maximal with respect to having finite intersection property. Conversely, if a family $\mathscr B$ of $Z(B_1(X))$ has finite intersection property and maximal with respect to having this property, then $\mathscr B$ is a $Z_B$-ultrafilter on $X$.
\end{remark}
\noindent In what follow, by an ideal $I$ of $B_1(X)$ we always mean a proper ideal.
\begin{theorem}\label{p2thm2.5}
	If $I$ is an ideal in $B_1(X)$, then $Z_B[I]=\{Z(f):f \in I\}$ is a $Z_B$-filter on $X$.
\end{theorem}
\begin{proof}
	Since $I$ is a proper ideal in $B_1(X)$, we claim $\emptyset \notin Z_B[I]$. If possible let $\emptyset \in Z_B[I]$. So, $\emptyset = Z(f)$, for some $f \in I$. As $f \in I \implies f^2 \in I$ and $Z(f^2)=Z(f)=\emptyset$, hence $\frac{1}{f^2} \in B_1(X)$ \cite{AA}. This is a contradiction to the fact that, $I$ is a proper ideal and contains no unit.\\
	Let $Z(f),Z(g) \in Z_B[I]$, for some $f,g \in I.$ Our claim is $Z(f) \cap Z(g) \in Z_B[I]$. $Z(f) \cap Z(g)=Z(f^2+g^2) \in Z_B[I]$, as $I$ is an ideal and so, $f^2+g^2 \in I$.\\
	Now assume that $Z(f)\in Z_B[I]$ and $Z' \in Z(B_1(X))$ is such that $Z(f) \subseteq Z'$. Then we can write $Z'=Z(h)$, for some $h \in B_1(X)$. $Z(f)\subseteq Z' \implies Z(h)=Z(h)\cup Z(f)$. So, $Z(h)=Z(hf)\in Z_B[I]$, because $hf \in I.$ Hence, $Z_B[I]$ is a $Z_B$-filter on $X$.
\end{proof}
\begin{theorem}\label{p2thm2.6}
	Let $\mathscr F$ be a $Z_B$-filter on $X$. Then $Z_B^{-1}[\mathscr F]= \{f \in B_1(X):Z(f) \in \mathscr F\}$ is an ideal in $B_1(X)$.
\end{theorem}
\begin{proof}
	We note that, $\emptyset \notin \mathscr F$. So the constant function \textbf{1} $\notin Z_B^{-1}[\mathscr F]$. Hence $Z_B^{-1}[\mathscr F]$ is a proper subset of $B_1(X)$.\\
	Choose $f,g \in Z_B^{-1}[\mathscr F]$. Then $Z(f), Z(g) \in \mathscr F$ and $\mathscr F$ being a $Z_B$-filter $Z(f)  \cap  Z(g) \in \mathscr F$. Now $Z(f) \cap Z(g) \subseteq Z(f-g)$. Hence $Z(f-g) \in \mathscr F$, $\mathscr F$ being a $Z_B$-filter on $X$. This implies $f-g \in Z_B^{-1} [\mathscr F]$.\\
	For $f \in Z_B^{-1}[\mathscr F]$ and $h \in B_1(X)$, $Z(f.h)=Z(f) \cup Z(h)$. As $Z(f) \in \mathscr F$ and $\mathscr F$ is a $Z_B$-filter on $X$, it follows that $Z(f.h)$ $\in$ $\mathscr F$. Hence $f.h \in Z_B^{-1}[\mathscr F]$.\\ 
Thus $Z_B^{-1}[\mathscr F]$ is an ideal of $B_1(X)$.
\end{proof}
 
\noindent We may define a map $Z: B_1(X) \rightarrow Z(B_1(X))$ given by  $f \mapsto Z(f)$. Certainly, $Z$ is a surjection. In view of the above results, such $Z$ induces a map $Z_B$ between the collection of all ideals of $B_1(X)$, say $\mathscr I_B$ and the collection of all $Z_B$-filters on $X$, say $\mathscr F_B(X) $, i.e., $Z_B : \mathscr I_B \rightarrow \mathscr F_B(X)$ given by $Z_B(I)= Z_B[I]$, $\forall$ $I \in \mathscr I_B$.
The map $Z_B$ is also a surjective map because for any $\mathscr F \in \mathscr F_B(X)$, $Z_B^{-1}[\mathscr F]$ is an ideal in $B_1(X)$. We also note that $Z_B[Z_B^{-1}[\mathscr F]]= \mathscr F$. So each $Z_B$-filter on $X$ is the image of some ideal in $B_1(X)$ under the map $Z_B: \mathscr I_B \rightarrow \mathscr F_B(X)$.\\\\
\textbf{Observation.} The map $Z_B: \mathscr I_B \rightarrow \mathscr F_B(X)$ is not injective in general. Because, for any ideal $I$ in $B_1(X)$, $Z_B^{-1}[Z_B[I]]$ is an ideal in $B_1(X)$, such that $I \subseteq Z_B^{-1}[Z_B[I]]$ and by our previous result $Z_B[Z_B^{-1}[Z_B[I]]]=Z_B[I]$. If one gets an ideal $J$ in $B_1(X)$ such that $I \subseteq J \subseteq Z_B^{-1}[Z_B[I]]$, then we must have $Z_B[I]=Z_B[J]$. The following example shows that such an ideal is indeed possible to exist. In fact, in the following example, we get countably many ideals $I_n$ in $B_1(\mathbb{R})$ such that the images of all the ideals are same under the map $Z_B$.
\begin{example}
	Let $f_0: \mathbb{R} \rightarrow \mathbb{R}$ be defined as, \\ 
	\[ f_0(x)= \begin{cases} 
	\frac{1}{q} &$ if $x= \frac{p}{q}$, where $ p \in \mathbb{Z}, q \in \mathbb{N}$ and g.c.d. $(p,q)=1 \\
1 & $ if $x= 0 $ $\\
	
		0 & $ otherwise  $
	\end{cases}
	\]\\
It is well known that $f_0 \in B_1(\mathbb{R})$ (see \cite{JPF}). Consider the ideal $I$ generated by $f_0$, i.e., $I=\langle f_0 \rangle$. We claim that $f_0^\frac{1}{3} \notin I$. If possible, let $f_0^\frac{1}{3} \in I$. Then there exists $g \in B_1(\mathbb{R})$, such that $f_0^\frac{1}{3}=gf_0$. When $x=\frac{p}{q}$,  where $ p \in \mathbb{Z}, q \in \mathbb{N}$ and g.c.d $(p,q)=1$, $g(x)=q^\frac{2}{3}$. We show that such $g$ does not exist in $B_1(\mathbb{R})$.\\
	Let $\alpha$ be any irrational number in $\mathbb{R}$. We show that $g$ is not continuous at $\alpha$, no matter how we define $g(\alpha)$. Suppose $g(\alpha)= \beta$. There exists a sequence of rational numbers $\{ \frac{p_m}{q_m}\}$, such that $\{ \frac{p_m}{q_m}\}$ converges to $\alpha$ and $ p_m \in \mathbb{Z}, q_m \in \mathbb{N}$ with g.c.d $(p_m,q_m)=1$, $\forall$ $m \in \mathbb{N}$. If $g$ is continuous at $\alpha$ then $\{g( \frac{p_m}{q_m})\}$ converges to $g(\alpha)$, which implies that $q_m^\frac{2}{3}$ converges to $\beta$. But $q_m \in \mathbb{N}$, so $\{q_m^\frac{2}{3}\}$ must be eventually constant. Suppose there exists $n_0 \in \mathbb{N}$ such that $\forall$ $m \geq n_0$, $q_m$ is either $c$ or $-c$ or $q_m$ oscillates between $c$ and $-c$, for some natural number $c$, i.e., $\{\frac{p_m}{c}\}$ converges to $\alpha$ or $-\alpha$ or oscillates. In any case, $\{\frac{p_m}{q_m}\}$ cannot converges to $\alpha$. Hence we get a contradiction. So, $g$ is not continuous at any irrational point. It is well known that, if, $f \in B_1(X,Y)$, where $X$ is a Baire space, $Y$ is a metric space and $B_1(X,Y)$ stands for the collection of all Baire one functions from $X$ to $Y$ then the set of points where $f$ is continuous is dense in $X$ \cite{JRM}. Therefore, the set of points of $\mathbb{R}$ where $g$ is continuous is dense in $\mathbb{R}$ and is a subset of $\mathbb{Q}$. Hence it is a countable dense subset of $\mathbb{R}$ (Since $\mathbb{R}$ is a Baire space). But using Baire's category theorem it can be shown that, there exists no function $f:\mathbb{R} \rightarrow \mathbb{R}$, which is continuous precisely on a countable dense subset of $\mathbb{R}$. So, we arrive at a contradiction and no such $g$ exists. Hence $f_0^\frac{1}{3} \notin I$.\\
Observe that, $Z(f_0)=Z(f_0^{\frac{1}{3}})$ and $I \subseteq Z_B^{-1}[Z_B[I]]$. Again, $f_0^{\frac{1}{3}} \notin I $ but $f_0^{\frac{1}{3}} \in Z_B^{-1}Z_B[I] $, which implies $I \subsetneqq Z_B^{-1}[Z_B[I]]$. Hence by an earlier result $Z_B[I]=Z_B[Z_B^{-1}[Z_B[I]]]$, proving that the map $Z_B: \mathscr I_B \rightarrow \mathscr F_B(X)$ is not injective when $X=\mathbb{R}$.
\end{example}
\noindent \textbf{Observation:} $\langle f_0 \rangle \subsetneq \langle f_0^\frac{1}{3} \rangle$. Analogously, it can be shown that $\langle f_0 \rangle \subsetneq \langle f_0^\frac{1}{3} \rangle \subsetneq \langle f_0^\frac{1}{5} \rangle \subsetneq...\subsetneq \langle f_0^\frac{1}{2m+1}\rangle  \subsetneq...$ is a strictly increasing chain of proper ideals in $B_1(\mathbb{R})$. Hence $B_1(\mathbb{R})$ is not a Noetherian ring.
\begin{theorem}\label{p2thm2.8}
	If $M$ is a maximal ideal in $B_1(X)$ then $Z_B[M]$ is a $Z_B$-ultrafilter on $X$.
\end{theorem}
\begin{proof}
	By Theorem \ref{p2thm2.5}, $Z_B[M]$ is a $Z_B$-filter on $X$. Let $\mathscr F$ be a $Z_B$-filter on $X$ such that, $Z_B[M]\subseteq \mathscr F$. Then $M \subseteq Z_B^{-1}[Z_B[M]]\subseteq Z_B^{-1}[\mathscr F]$. $Z_B^{-1}[\mathscr F]$ being a proper ideal and $M$ being a maximal ideal, we have $Z_B^{-1}[\mathscr F]=M \implies Z_B[M]=Z_B[Z_B^{-1}[\mathscr F]]= \mathscr F$. Hence every $Z_B$-filter that contains $Z_B[M]$ must be equal to $Z_B[M]$. This shows $Z_B[M]$ is a $Z_B$-ultrafilter on $X$.
\end{proof}
\begin{theorem}
	If $\mathscr U$ is a $Z_B$-ultrafilter on $X$ then $Z_B^{-1}[\mathscr U]$ is a maximal ideal in $B_1(X)$.
\end{theorem}
\begin{proof}
	By Theorem \ref{p2thm2.6}, we have $Z_B^{-1}[\mathscr U]$ is a proper ideal in $B_1(X)$. Let $I$ be a proper ideal in $B_1(X)$ such that $Z_B^{-1}[\mathscr U] \subseteq I$. It is enough to show that $Z_B^{-1}[\mathscr U] = I$. Now $Z_B^{-1}[\mathscr U] \subseteq I \implies Z_B[Z_B^{-1}[\mathscr U]] \subseteq Z_B[I] \implies \mathscr U \subseteq Z_B[I]$. Since $\mathscr U$ is a $Z_B$-ultrafilter on $X$, we have $\mathscr U = Z_B[I] \implies Z_B^{-1}[\mathscr U]=Z_B^{-1}[Z_B[I]] \supseteq I$. Hence $Z_B^{-1}[\mathscr U] = I$
\end{proof}
\begin{remark}
	Each $Z_B$-ultrafilter on $X$ is the image of a maximal ideal in $B_1(X$) under the map $Z_B$.
\end{remark}
\noindent Let $\mathcal M(B_1(X))$ be the collection of all maximal ideals in $B_1(X$) and $\Omega_B(X)$ be the collection of all $Z_B$-ultrafilters on $X$. If we restrict the map $Z_B$ to the class $\mathcal M(B_1(X))$, then it is clear that the map $Z_B\bigg|_{\mathcal M(B_1(X))}:\mathcal M(B_1(X)) \rightarrow \Omega_B(X)$ is a surjective map. Further, this restriction map is a bijection, as seen below. 
\begin{theorem}
	The map $Z_B\bigg|_{\mathcal M(B_1(X))}:\mathcal M(B_1(X)) \rightarrow \Omega_B(X)$ is a bijection.
\end{theorem}
\begin{proof}
	It is enough to check that $Z_B\bigg|_{\mathcal M(B_1(X))}: \mathcal M(B_1(X)) \rightarrow \Omega_B(X)$ is injective. Let $M_1$ and $M_2$ be two members in $\mathcal M(B_1(X))$ such that $ Z_B[M_1]=Z_B[M_2] \implies Z_B^{-1}[Z_B[M_1]]=Z_B^{-1}[Z_B[M_1]]$. But $M_1 \subseteq Z_B^{-1}[Z_B[M_1]]$ and $M_2 \subseteq Z_B^{-1}[Z_B[M_2]]$. By maximality of $M_1$ and $M_2$ we have, $M_1 = Z_B^{-1}[Z_B[M_1]] = Z_B^{-1}[Z_B[M_2]] =M_2 $.
\end{proof}
\begin{definition}
	An ideal $I$ in $B_1(X)$ is called a $Z_B$-ideal if $Z_B^{-1}[Z_B[I]]=I$, i.e., $\forall$ $f \in B_1(X)$, $f \in I \iff Z(f) \in Z_B[I]$.
\end{definition}
\noindent Since $Z_B[Z_B^{-1}[\mathscr F_B]]= \mathscr F_B$, $Z_B^{-1}[\mathscr F_B]$ is a $Z_B$-ideal for any $Z_B$-filter $\mathscr F_B$ on $X$.\\
If $I$ is any ideal in $B_1(X)$, then, $Z_B^{-1}[Z_B[I]]$ is the smallest $Z_B$-ideal containing $I$. It is easy to observe 
\begin{enumerate}
	\item Every maximal ideal in $B_1(X)$ is a $Z_B$ ideal.
	\item The intersection of arbitrary family of $Z_B$-ideals in $B_1(X)$ is always a $Z_B$-ideal.
	\item The map $Z_B \bigg |_{\mathscr J_B}: \mathscr J_B \rightarrow \mathscr F_B(X)$  is a bijection, where $\mathscr J_B$ denotes the collection of all $Z_B$-filters on $X$.
\end{enumerate}

\begin{example}
	Let $I=\{f \in B_1(\mathbb{R}):f(1)=f(2)=0\}$. Then $I$ is a $Z_B$ ideal in $B_1(\mathbb{R})$ which is not maximal, as $I \subsetneq \widehat{M}_{1}=\{f \in B_1(\mathbb{R}):f(1)=0\}$, which is in fact a maximal ideal. The ideal $I$ is not a prime ideal, as the functions $x-1$ and $x-2$ do not belong to $I$, but their product belongs to $I$. Also no proper ideal of $I$ is prime. More generally, for any subset $S$ of $\mathbb{R}, I_{S}=\{f \in B_1(\mathbb{R}):f(S)=0\}$ is a $Z_B$-ideal in $B_1(\mathbb{R})$.
\end{example}
\noindent It is well known that in a commutative ring $R$ with unity, the intersection of all prime ideals of $R$ containing an ideal $I$ is called the \textbf{radical of $I$} and it is denoted by $\sqrt{I}$. For any ideal $I$, the radical of $I$ is given by $\{a \in R: a^n\in I$, for some $n \in \mathbb{N}\}$ (\cite{GJ}) and in general $I \subseteq \sqrt{I}$. For if $I= \sqrt{I}$, $I$ is called a radical ideal.
\begin{theorem}
	A $Z_B$-ideal $I$ in $B_1(X)$ is a radical ideal.
\end{theorem}
\begin{proof}
	$\sqrt{I}=\{f \in B_1(X): \exists$ $n \in \mathbb{N}$ such that $f^n \in I\}=\{f \in B_1(X): $ such that $Z(f^n) \in Z_B[I]$ for some $n \in \mathbb{N}\}$ \big(As $I$ is a $Z_B$-ideal in $B_1(X)$ \big) $=\{f \in B_1(X): $ such that $Z(f) \in Z_B[I]$ $\}=\{f \in B_1(X): f \in I\}=I$. So $I$ is a radical ideal in $B_1(X)$.
\end{proof}
\begin{corollary}
	Every $Z_B$-ideal $I$ in $B_1(X)$ is the intersection of all prime ideals in $B_1(X)$ which contains $I$.
	 
\end{corollary}
\noindent Next theorem establishes some equivalent conditions on the relationship among $Z_B$-ideals and prime ideals of $B_1(X)$.
\begin{theorem}\label{eqv4}
	For a $Z_B$-ideal $I$ in $B_1(X)$ the following conditions are equivalent:
	\begin{enumerate}
		\item $I$ is a prime ideal of $B_1(X)$.
		\item $I$ contains a prime ideal of $B_1(X)$.
		\item if $fg=0$, then either $f \in I$ or $g \in I$.
		\item Given $f \in B_1(X)$ there exists $Z \in Z_B[I]$, such that $f$ does not change its sign on $Z$.
	\end{enumerate}
\end{theorem}
\begin{proof}
	$(1) \implies (2)$ and 	$(2) \implies (3)$ are immediate.\\
	$(3) \implies (4)$:
	Let $(3)$ be true. Choose $f \in B_1(X).$ Then $(f \vee 0). (f \wedge 0)=0$. So by $(3)$, $f \vee 0 \in I$ or $f \wedge 0 \in I$. Hence $Z(f \vee 0) \in Z_B[I]$ or $Z(f \wedge 0) \in Z_B[I]$. It is clear that $f \leq 0$ on $Z(f \wedge 0)$ and $f \geq 0$ on $Z(f \vee 0)$.\\\\
	$(4) \implies (1)$: Let $(4)$ be true. To show that $I$ is prime. Let $g,h \in B_1(X)$ be such that $gh \in I$. By $(4)$ there exists $Z \in Z_B[I]$, such that $|g|-|h| \geq 0$ on $Z$ (say). It is clear that, $Z \cap Z(g) \subseteq Z(h)$. Consequently $Z \cap Z(gh) \subseteq Z(h)$. Since $Z_B[I]$ is a $Z_B$-filter on $X$, it follows that $Z(h) \in Z_B[I]$. So $h \in I $, since $I$ is a $Z_B$-ideal. Hence, $I$ is prime.
\end{proof}

\begin{theorem}
	In $B_1(X)$, every prime ideal $P$ can be extended to a unique maximal ideal.
\end{theorem}
\begin{proof}
	If possible let $P$ be contained in two distinct maximal ideals $M_1$ and $M_2$. So, $P \subseteq M_1 \cap M_2$. Since maximal ideals in $B_1(X)$ are $Z_B$-ideals and intersection of any number of $Z_B$-ideals is $Z_B$-ideal, $M_1 \cap M_2$ is a $Z_B$-ideal containing the prime ideal $P$. By Theorem \ref{eqv4}, $M_1 \cap M_2$ is a prime ideal. But in a commutative ring with unity, for two ideals $I$ and $J$, if, $I \nsubseteq J$ and $J \nsubseteq I$, then $I \cap J $ is not a prime ideal. Thus $M_1 \cap M_2$ is not prime ideal and we get a contradiction. So, every prime ideal can be extended to a unique maximal ideal.
\end{proof}
\begin{corollary}
	$B_1(X)$ is a Gelfand ring for any topological space $X$.
\end{corollary}
\begin{definition}
	A $Z_B$-filter $\mathscr F _B$ on $X$ is called a prime $Z_B$-filter on $X$, if, for any $Z_1,Z_2 \in Z(B_1(X))$ with $Z_1 \cup Z_2 \in \mathscr F_B$ either $Z_1 \in \mathscr F_B$ or $Z_2 \in \mathscr F_B$.
\end{definition}
\noindent The next two theorems are analogous to Theorem 2.12 in \cite{GJ} and therefore, we state them without proof.
\begin{theorem}
		If $I$ is a prime ideal in $B_1(X)$, then $Z_B[I]=\{Z(f):f \in I\}$ is a prime $Z_B$-filter on $X$.
\end{theorem}
\begin{theorem}
		Let $\mathscr F_B$ be a prime $Z_B$-filter on $X$ then $Z_B^{-1}[\mathscr F_B]= \{f \in B_1(X):Z(f) \in \mathscr F_B\}$ is a prime ideal in $B_1(X)$.
\end{theorem}
\begin{corollary}
	Every prime $Z_B$-filter can be extended to a unique $Z_B$-ultrafilter on $X$.
\end{corollary}
\begin{corollary}
	A $Z_B$-ultrafilter $\mathscr U$ on $X$ is a prime $Z_B$-filter on $X$, as $\mathscr U= Z_B[M]$, for some maximal ideal $M$ in $B_1(X)$.
\end{corollary}

\section{Fixed ideals and free ideals in $B_1(X)$}
\noindent In this section, we introduce fixed and free ideals of $B_1(X)$ and $B_1^*(X)$ and completely characterize the fixed maximal ideals of $B_1(X)$ as well as those of $B_1^*(X)$. It is observed here that a natural relationship exists between fixed maximal ideals of $B_1^*(X)$ and the fixed maximal ideals of $B_1(X)$. However, free maximal ideals do not behave the same. In the last part of this section, we find a class of Baire one functions defined on a perfectly normal $T_1$ space $X$ which precisely determine the fixed and free maximal ideals of the corresponding ring.   

\begin{definition}
A proper ideal $I$ of $B_1(X)$ (respectively, $B_1^*(X)$) is called \textbf{fixed} if $ \bigcap Z[I] \neq \emptyset$. If $I$ is not fixed then it is called \textbf{free}.
\end{definition}
 \noindent For any Tychonoff space $X$, the fixed maximal ideals of the ring $B_1(X)$ and those of $B_1^*(X)$ are characterized.
\begin{theorem}\label{fixed_max}
	$\{\widehat{M}_p: p \in X\}$ is a complete list of fixed maximal ideals in $B_1(X)$, where $\widehat{M}_p= \{f \in B_1(X): f(p)=0\}$. Moreover, $p \neq q \implies \widehat{M}_p \neq \widehat{M}_q$.
\end{theorem}
\begin{proof}
Choose $p \in X$. The map $\Psi_p: B_1(X) \rightarrow \mathbb{R}$, defined by $\Psi(f)= f(p)$ is clearly a ring homomorphism. Since the constant functions are in $B_1(X)$, $\Psi_p$ is surjective and $\ker \Psi_p = \{f \in B_1(X): \Psi_p(f)=0\}=\{f \in B_1(X): f(p)=0\}=\widehat{M}_p$ (say).\\
By First isomorphism theorem of rings we get $B_1(X)/\widehat{M}_p$ is isomorphic to the field $\mathbb{R}$. $B_1(X)/\widehat{M}_p$ being a field we conclude that $\widehat{M}_p$ is a maximal ideal in $B_1(X)$. Since $p \in \bigcap Z_B[M]$, the ideal $\widehat{M}_p$ is a fixed ideal.\\
For any Tychonoff space $X$, we know that $p \neq q \implies M_p \neq M_q$, where $M_p=\{f \in C(X): f(p)=0\}$ is the fixed maximal ideal in $C(X)$. Since $\widehat{M_p} \cap C(X) = M_p$, it follows that for any Tychonoff space $X$, $p \neq q \implies \widehat{M_p} \neq \widehat{M_q}$. \\
Let $M$ be any fixed maximal ideal in $B_1(X)$. There exists $p \in X$ such that for all $f \in B_1(X),$ $f(p)=0$. Therefore, $M \subseteq \widehat{M_p}$. Since $M$ is a maximal ideal and $\widehat{M_p}$ is a proper ideal, we get $M=\widehat{M_p}$.
\end{proof}
\begin{theorem}
$\{\widehat{M}^*_p: p \in X\}$ is a complete list of fixed maximal ideals in $B_1(X)$, where $\widehat{M}^*_p= \{f \in B_1^*(X): f(p)=0\}$. Moreover, $p \neq q \implies \widehat{M}^*_p \neq \widehat{M}^*_q$.	
\end{theorem}
\begin{proof}
	Similar to the proof of Theorem~\ref{fixed_max}.
\end{proof}
\noindent The following two theorems show the interrelations between fixed ideals of $B_1(X)$ and $B_1^*(X)$.  
	\begin{theorem}
	 If $I$ is any fixed ideal of $B_1(X)$ then $I \cap B_1^*(X)$ is a fixed ideal of $B_1^*(X)$.
	\end{theorem}
	\begin{proof}
		Straightforward.
	\end{proof}
\begin{lemma}
	Given any $f \in B_1(X)$, there exists a positive unit $u$ of $B_1(X)$ such that $uf \in B_1^*(X)$.
\end{lemma}
\begin{proof}
	Consider $u= \frac{1}{|f|+1}$. Clearly $u$ is a positive unit in $B_1(X)$ \cite{AA} and $uf \in B_1^*(X)$ as $|uf| \leq 1$.
\end{proof}
\begin{theorem}
Let an ideal $I$ in $B_1(X)$ be such that $I \cap B_1^*(X)$ is a fixed ideal of $B_1^*(X)$. Then $I$ is a fixed ideal of $B_1(X)$.
\end{theorem} 
\begin{proof}
For each $f \in I$, there exists a positive unit $u_f$ of $B_1(X)$ such that $u_f f \in I \cap B_1^*(X)$. Therefore, $\bigcap\limits_{f \in I} Z(f)= \bigcap \limits_{f \in I} Z(u_f f) \supseteq \bigcap \limits_{g \in B_1^*(X)\cap I}Z(g) \neq \emptyset$. Hence $I$ is fixed in $B_1(X)$.
\end{proof}
\noindent Since for any discrete space $X$, $C(X)=B_1(X)$ and $C^*(X)=B_1^*(X)$, then considering the example 4.7 of \cite{GJ}, we can conclude the following:
\begin{enumerate}
	\item For any maximal ideal $M$ of $B_1(X)$, $M\cap B_1^*(X)$ need not be a maximal ideal in $B_1^*(X)$.
	\item All free maximal ideals in $B_1^*(X)$ need not be of the form $M\cap B_1^*(X)$, where $M$ is a maximal ideal in $B_1(X)$.
\end{enumerate}
\begin{theorem} \label{chiP}
If $X$ is a perfectly normal $T_1$ space then for each $p \in X,$ $\chi_p:X \rightarrow \mathbb{R}$ given by
	\[ \chi_p(x)= \begin{cases} 
	1 &$ if $x=p$$ \\

	0 & $ otherwise.  $	\end{cases}
	\]
	is a Baire one function.
\end{theorem}
\begin{proof}
For any open set $U$ of $\mathbb{R}$, 
	\[ \chi_p^{-1}(U)= \begin{cases} 
	X &$ if $0,1 \in U$$ \\
	
	X \setminus \{p\} & $ if $0 \in U$ but $1 \notin U$$\\
	\{p\}  & $ if $0 \notin U$ but $1 \in U$$\\
	\emptyset & $ if $0 \notin U$ but $1 \notin U$$.

	\end{cases}
	\]
	Since $X$ is a perfectly normal space, the open set $X \setminus \{p\}$ is a $F_{\sigma}$ set. Hence in any case  $\chi_p$ pulls back an open set to a $F_{\sigma}$ set. So $\chi_p$ is a Baire one function.
\end{proof}
\noindent In view of Theorem \ref{chiP} we obtain the following facts about any perfectly normal $T_1$ space.
\begin{obs}\label{Obs}
If $M$ is a maximal ideal of $B_1(X)$ where $X$ is a perfectly normal $T_1$ space then
\begin{enumerate}
\item For each $p \in X$ either $\chi_p \in M$ or $\chi_p - 1 \in M$.\\
 This follows from $\chi_p (\chi_p -1)=0 \in M$ and $M$ is prime.
 \item If $\chi_p -1 \in M$ then $\chi_q \in M$ for all $q \neq p$.\\
 For if $\chi_q-1 \in M$ for some $q \neq p$ then $Z(\chi_p-1), Z(\chi_q-1) \in Z_B[M]$. This implies $\emptyset=Z(\chi_p-1) \cap Z(\chi_q-1) \in Z_B[M]$ which contradicts that $Z_B[M]$ is a $Z_B$-ultrafilter.
 \item $M$ is fixed if and only if $\chi_p-1 \in M$ for some $p \in X$.\\
 If $M$ is fixed then $M=\widehat{M_p}$ for some $p \in X$ and therefore, $\chi_p-1 \in M$. Conversely let $\chi_p-1 \in M$ for some $p \in X$. Then $\{p\}=Z(\chi_p-1) \in Z_B[M]$ shows that $M$ is fixed.
 \item $M$ is free if and only if $M$ contains $\{\chi_p : p \in X\}$.\\
 Follows from Observation (3).
\end{enumerate}
\end{obs}
\noindent The following theorem ensures the existence of free maximal ideals in $B_1(X)$ where $X$ is any infinite perfectly normal $T_1$ space. 
\begin{theorem}
	For a perfectly normal space $T_{1}$ space $X$, the following statements are equivalent:
	\begin{enumerate}
		\item $X$ is finite.
		\item Every maximal ideal in $B_1(X)$ is fixed.
		\item Every ideal in $B_1(X)$ is fixed.
	\end{enumerate}
\end{theorem}
\begin{proof}
	$(1) \implies (2)$: Since a finite $T_{1}$ space is discrete, $C(X)=B_1(X)=X^{\mathbb{R}}$. $X$ being finite, it is compact and therefore all the maximal ideals of $C(X)$ $\big(=B_1(X) \big)$ are fixed.\\
	$(2) \implies (3)$: Proof obvious.\\
	$(3) \implies (1)$: Suppose $X$ is infinite. We shall show that there exists a free (proper) ideal in $B_1(X)$.\\
	Consider $I= \{f \in B_1(X): \overline{X\setminus Z(f)}$ is finite$\}$ (Here finite includes $\emptyset$).\\
	Of course $I \neq \emptyset$, as $\mathbf{0} \in I$. Since $X$ is infinite, $\mathbf{1} \notin I$ and so, $I$ is proper. We show that, $I$ is an ideal in $B_1(X)$. Let $f,g \in  I$. Then $\overline{X \setminus Z(f)}$ and $\overline{X \setminus Z(g)}$ are both finite. Now $\overline{X \setminus Z(f-g)} \subseteq \overline{X \setminus Z(f)} \cup \overline{X \setminus Z(g)}$ implies that $\overline{X \setminus Z(f-g)}$ is finite. Hence $f-g \in I$. Similarly, $\overline{X \setminus Z(f.g)} \subseteq \overline{X \setminus Z(f)}$ for any $f \in I$ and $g \in B_1(X)$. So, $\overline{X \setminus Z(f.g)}$ is finite and hence $f.g \in I$. Therefore, $I$ is an ideal in $B_1(X)$. We claim that $I$ is free.\\
	For any $p \in X$, consider $\chi_p:X \rightarrow \mathbb{R}$ given by
	\[ \chi_p(x)= \begin{cases} 
	1 &$ if $x=p$$ \\

	0 & $ otherwise.  $	\end{cases}
	\]
	Using Theorem \ref{chiP}, $\chi_p$ is a Baire one function. Also, $\overline{X \setminus Z(\chi_p)}= \overline{X \setminus(X\setminus \{p\})} = \overline{\{p\}}= \{p\}=$ finite and $\chi_p(p) \neq 0$. Hence, $I$ is free.
\end{proof}

\section{Residue class ring of $B_1(X)$ modulo ideals} 
\noindent An ideal $I$ in a partially ordered ring $A$ is called \textbf{convex} if for all $a,b,c \in A$ with $a \leq b \leq c$ and $c, a \in I \implies b \in I$. Equivalently, for all $a,b \in A, 0\leq a \leq b$ and $b \in I \implies a \in I$.\\
\noindent If $A$ is a lattice ordered ring then an ideal $I$ of $A$ is called \textbf{absolutely convex} if for all $a,b \in A$, $|a| \leq |b|$ and $b \in I \implies a \in I$.

\begin{example}
	If $t: B_1(X) \rightarrow B_1(Y)$ is a ring homomorphism, then $\ker t$ is an absolutely convex ideal.
\end{example}
\begin{proof}
	Let $g \in \ker t$ and $|f|\leq |g|$, where $f \in B_1(X)$. $g \in \ker t \implies t(g)=0 \implies t(|g|)=|t(g)|=0$. Since any ring homomorphism $t: B_1(X) \rightarrow B_1(Y)$ preserves the order, $t(|f|)=0 \implies |t(f)|=0 \implies t(f)=0 \implies f \in \ker t$.
\end{proof}
\noindent Let $I$ be an ideal in $B_1(X)$. In what follows we shall denote any member of the quotient ring $B_1(X)/I$ by $I(f)$ for $f \in B_1(X)$. i.e., $I(f)= f+I$. Now we begin with two well known theorems.
\begin{theorem} \cite{GJ} \label{orderdef}
	Let $I$ be an ideal in a partially ordered ring $A$. The corresponding quotient ring $A/I$ is a partially ordered ring if and only if $I$ is convex, where the partial order is given by
	$
	I(a) \geq 0 \ \textnormal{iff} \ \exists$ $x \in A \ \textnormal{ such that } x \geq 0 \textnormal{ and } a \equiv x (\mod I).
	$
\end{theorem}
\begin{theorem} \cite{GJ} \label{equiv}
	On a convex ideal $I$ in a lattice-ordered ring $A$ the following conditions are equivalent.
	\begin{enumerate}
		\item $I$ is absolutely convex.
		\item $x \in I$ implies $|x| \in I.$
		\item $x, y \in I$ implies $x \vee y \in I.$
		\item $I(a \vee b)= I(a) \vee I(b)$, whence $A/I$ is a lattice ordered ring.
		\item $\forall$ $a \in A$, $I(a) \geq 0$ iff $I(a)=I(|a|)$.
		\end{enumerate}
\end{theorem}
\begin{remark}
	For an absolutely convex ideal $I$ of $A$, $I(|a|)=I(a \vee -a)=I(a) \vee I(-a)=|I(a)|$, $\forall$ $a \in A$.
\end{remark}
\begin{theorem}
	Every $Z_B$-ideal in $B_1(X)$ is absolutely convex.
\end{theorem}
\begin{proof}
	Suppose $I$ is any $Z_B$-ideal and $|f| \leq |g|$, where $g \in I$ and $f \in B_1(X)$. Then $Z(g) \subseteq Z(f)$. Since $g \in I$, it follows that $Z(g) \in Z_B[I]$, hence $Z(f) \in Z_B[I]$. Now $I$ being a $Z_B$-ideal, $f \in I$.
\end{proof}
\begin{corollary}
	In particular every maximal ideal in $B_1(X)$ is absolutely convex.
\end{corollary} 
\begin{theorem}
	For every maximal ideal $M$ in $B_1(X)$, the quotient ring $B_1(X)/M$ is a lattice ordered field.
\end{theorem}
\begin{proof}
	Proof is immediate.
\end{proof}
\noindent The following theorem is a characterization of the non-negative elements in the lattice ordered ring $B_1(X)/I$, where $I$ is a $Z_B$-ideal.
\begin{theorem} \label{alter}
Let $I$ be a $Z_B$-ideal in $B_1(X)$. For $f \in B_1(X)$, $I(f) \geq 0$ if and only if there exists $Z \in Z_B[I]$ such that $f \geq 0 $ on $Z$.
\end{theorem}
\begin{proof}
Let $I(f) \geq 0$. By condition $(5)$ of Theorem \ref{equiv}, we write $I(f)=I(|f|)$. So, $f-|f| \in I \implies Z(f-|f|) \in Z_B[I]$ and $f \geq 0$ on $Z(f-|f|)$.\\
Conversely, let $f \geq 0$ on some  $Z \in Z_B[I]$. Then $f=|f|$ on $Z \implies Z \subseteq Z(f-|f|) \implies Z(f-|f|) \in Z_B[I]$. $I$ being a $Z_B$-ideal we get $f-|f| \in I$, which means $I(f)=I(|f|)$. But $|f| \geq 0$ on $Z$ gives $I(|f|) \geq 0$. Hence, $I(f) \geq 0$.
\end{proof}
\begin{theorem} \label{strict}
Let $I$ be any $Z_B$-ideal and $f \in B_1(X)$. If there exists $Z \in Z_B[I]$ such that $f(x) > 0$, for all $x$ $\in Z$, then $I(f) >0$. 
\end{theorem}
\begin{proof}
	By Theorem~\ref{alter}, $I(f) \geq 0$. But $Z \cap Z(f) = \emptyset$ and $Z \in Z_B[I] \implies Z(f) \notin Z_B[I] \implies f \notin I \implies I(f) \neq 0 \implies I(f) >0$.
\end{proof}
\noindent The next theorem shows that the converse of the above theorem holds if the ideal is a maximal ideal in $B_1(X)$.
\begin{theorem}\label{next2strict}
	Let $M$ be any maximal ideal in $B_1(X)$ and $M(f) >0$ for some $f \in B_1(X)$ then there exists $Z \in Z_B[M]$ such that $ f >0 $ on $Z$.
\end{theorem}
\begin{proof}
	By Theorem \ref{alter}, there exists $Z_1 \in Z_B[M]$ such that $f \geq 0$ on $Z_1$. Now $M(f) >0 \implies f \notin M$ which implies that there exists $g$ $\in M$, such that $Z(f) \cap Z(g)= \emptyset$. Choosing $Z=Z_1 \cap Z(g)$, we observe $Z \in Z_B[M]$ and $f(x) >0$, for all $x \in Z$.
\end{proof}
\begin{corollary}
	For a maximal ideal $M$ of $B_1(X)$ and for some $f \in B_1(X)$, $M(f)>0$ if and only if there exists $Z \in Z_B[M]$ such that $f(x)>0$ on $Z$.
\end{corollary}
\noindent Now we show Theorem~\ref{next2strict} doesn't hold for every non-maximal ideal $I$.
\begin{theorem}
	Suppose $I$ is any non-maximal $Z_B$-ideal in $B_1(X)$. There exists $f \in B_1(X)$ such that $I(f)>0$ but $f$ is not strictly positive on any $Z \in Z_B[I]$.
\end{theorem}
\begin{proof}
	Since $I$ is non-maximal, there exists a proper ideal $J$ of $B_1(X)$ such that $I \subsetneqq J$. Choose $f \in J \smallsetminus I$. $f^2 \notin I \implies I(f^2) >0$. Choose any $Z \in Z_B[I]$. Certainly, $Z \in Z_B[J]$ and so, $Z \cap Z(f^2) \in Z_B[J] \implies Z \cap Z(f^2) \neq \emptyset$. So $f$ is not strictly positive on the whole of $Z$.
\end{proof}
\noindent In what follows, we characterize the ideals $I$ in $B_1(X)$ for which $B_1(X)/I$ is a totally ordered ring.
\begin{theorem}
	Let $I$ be a $Z_B$-ideal in $B_1(X)$, then the lattice ordered ring $B_1(X)/I$ is totally ordered ring if and only if $I$ is a prime ideal.
\end{theorem}
\begin{proof}
$B_1(X)/I$ is a totally ordered ring if and only if for any $f \in B_1(X)$, $I(f) \geq 0$ or $I(- f) \leq 0$ if and only if for all $f \in B_1(X)$, there exists $Z \in Z_B[I]$ such that $f$ does not change its sign on $Z$ if and only if $I$ is a prime ideal (by Theorem \ref{eqv4}).	
\end{proof}
\begin{corollary}
	 For every maximal ideal $M$ in $B_1(X)$, $B_1(X)/M$ is a totally ordered field.
\end{corollary}
\begin{theorem}
Let $M$ be a maximal ideal in $B_1(X)$. The function $\Phi: \mathbb{R} \rightarrow B_1(X)/M$ (respectively, $\Phi:\mathbb{R} \rightarrow B_1^*(X)/M$) defined by $\Phi(r)=M(\mathbf{r})$, for all $r \in \mathbb{R}$, where $\mathbf{r}$ denotes the constant function with value $r$, is an order preserving monomorphism.	
\end{theorem}
\begin{proof}
	It is clear from the definitions of addition and multiplication of the residue class ring $B_1(X)/M$ that the function is a homomorphism. \\
	To show $\phi$ is injective. Let $M(\mathbf{r})=M(\mathbf{s})$ for some $r,s \in \mathbb{R}$ with $r \neq s$. Then $\mathbf{r}-\mathbf{s} \in M$. This contradicts to the fact that $M$ is a proper ideal. Hence $M(\mathbf{r}) \neq M(\mathbf{s})$, when $r \neq s$.\\ 
	Let $r,s \in \mathbb{R}$ with $r >s$. Then $r-s >0$. The function $\mathbf{r}-\mathbf{s}$ is strictly positive on $X$. Since $X \in Z(B_1(X))$, by Theorem \ref{strict}, $M(\mathbf{r}-\mathbf{s})>0 \implies M(\mathbf{r}) > M(\mathbf{s}) \implies \Phi(r)> \Phi(s)$. Thus $\Phi$ is an order preserving monomorphism.
\end{proof}
\noindent For a maximal ideal $M$ in $B_1(X)$, the residue class field $B_1(X)/M$ (respectively $B_1^*(X)/M$) can be considered as an extension of the field $\mathbb{R}$.
\begin{definition}
The maximal ideal $M$ of $B_1(X)$ (respectively, $B_1^*(X)$) is called real if $\Phi(\mathbb{R})=B_1(X)/M$ (respectively, $\Phi(\mathbb{R})=B_1^*(X)/M$) and in such case $B_1(X)/M$ is called \textbf{real} residue class field. If $M$ is not real then it is called \textbf{hyper-real} and $B_1(X)/M$ is called hyper-real residue class field.
\end{definition}
\begin{definition} \cite{GJ}
	A totally ordered field $F$ is called \textbf{archimedean} if given $\alpha \in F$, there exists $n \in \mathbb{N}$ such that $n > \alpha$. If $F$ is not archimedean then it is called \textbf{non-archimedean}.
\end{definition}
\noindent If $F$ is a non-archimedean ordered field then there exists some $\alpha \in F$ such that $\alpha > n$, for all $n \in \mathbb{N}$. Such an $\alpha$ is called an infinitely large element of $F$ and $\frac{1}{\alpha}$ is called infinitely small element of $F$ which is characterized by the relation $0<\frac{1}{\alpha}<\frac{1}{n}$, $\forall$ $n \in \mathbb{N}$. The existence of an infinitely large (equivalently, infinitely small) element in $F$ assures that $F$ is non-archimedean.\\
\noindent In the context of archimedean field, the following is an important theorem available in the literature.
\begin{theorem} \cite{GJ}
	A totally ordered field is archimedean iff it is isomorphic to a subfield of the ordered field $\mathbb{R}$ .
\end{theorem}
\noindent We thus get that the real residue class field $B_1(X)/M$ is archimedean if $M$ is a real maximal ideal of $B_1(X)$.

\begin{theorem}
	Every hyper-real residue class field $B_1(X)/M$ is non-archimedean.
\end{theorem}
\begin{proof}
	Proof follows from the fact that the identity is the only non-zero homomorphism on the ring $\mathbb{R}$ into itself.
\end{proof}
\begin{corollary}
	A maximal ideal $M$ of $B_1(X)$ is hyper-real if and only if there exists $f \in B_1(X)$ such that $M(f)$ is an infinitely large member of $B_1(X)/M$.
\end{corollary}
\begin{theorem}
	Each maximal ideal $M$ in $B_1^*(X)$ is real. 
\end{theorem}
\begin{proof}
	It is equivalent to show that $B_1^*(X)/M$ is archimedean. Choose $f \in B_1^*(X)$. Then $|f(x)|\leq n$, for all $x \in X$ and for some $n \in \mathbb{N}$. i.e., $|M(f)|=M(|f|) \leq M(\mathbf{n})$. So there does not exist any infinitely large member in $B_1^*(X)/M$ and hence $B_1^*(X)/M$ is archimedean.
\end{proof}
\begin{corollary}
	If $X$ is a topological space such that $B_1(X)=B_1^*(X)$ then each maximal ideal in $B_1(X)$ is real.
\end{corollary}
\noindent The following theorem shows how an unbounded Baire one function $f$ on $X$ is related to an infinitely large member of the residue class field $B_1(X)/M$.
\begin{theorem} \label{3eqv}
	Given a maximal ideal $M$ of $B_1(X)$ and $f \in B_1(X)$, the following statements are equivalent:
	\begin{enumerate}
		\item $|M(f)|$ is infinitely large member in $B_1(X)/M$.
		\item $f$ is unbounded on each zero set in $Z_B[M]$.
		\item for all $n \in \mathbb{N}$, $Z_n=\{x \in X: |f(x)| \geq n\} \in Z_B[M]$.
	\end{enumerate}
\end{theorem}
\begin{proof}
	$(1) \iff (2)$: $|M(f)|$ is not infinitely large in $B_1(X)/M$ iff $\exists$ $n \in \mathbb{N}$ such that, $|M(f)|=M(|f|) \leq M(\mathbf{n})$ iff $|f| \leq \mathbf{n}$ on some $Z \in Z_B[M]$ if and only if $f$ is bounded on some $Z \in Z_B[M]$. \\
	$(2) \implies (3)$: Choose $n \in \mathbb{N}$, we shall show that $Z_n \in Z_B[M]$. By $(2)$, $Z_n$ intersects each member in $Z_B[M]$. Now $Z_B[M]$ being a $Z_B$-ultrafilter, $Z_n \in Z_B[M]$.\\
	$(3) \implies (2)$: Let each $Z_n \in Z_B[M]$, for all $n \in \mathbb{N}$. Then for each $n \in \mathbb{N}$, $|f| \geq n$ on some zero set in $Z_B[M]$. Hence $|M(f)|=M(|f|) \geq M(\mathbf{n})$, for all $n \in \mathbb{N}$. That means $|M(f)|$ is infinitely large member in $B_1(X)/M$.
\end{proof}
\begin{theorem}
	$f \in B_1(X)$ is unbounded on $X$ if and only if there exists a maximal ideal $M$ in $B_1(X)$ such that $M(f)$ is infinitely large in $B_1(X)/M$.
\end{theorem}
\begin{proof}
	Let $f$ be unbounded on $X$. So, each $Z_n$ in Theorem \ref{3eqv} is non-empty. We observe that $\{Z_n: n \in \mathbb{N}\}$ is a subcollection of $Z(B_1(X))$ having finite intersection property. So there exists a $Z_B$-ultrafilter $\mathscr U$ on $X$ such that $\{Z_n: n \in \mathbb{N}\} \subseteq \mathscr U$. Therefore, there is a maximal ideal $M$ in $B_1(X)$ for which $\mathscr = Z_B[M]$ and so, $Z_n \in Z_B[M]$, for all $n \in \mathbb{N}$. By Theorem \ref{3eqv} $M(f)$ is infinitely large.\\
	Converse part is a consequence of $(1) \implies (2)$ of Theorem \ref{3eqv}.
\end{proof}
\begin{corollary}
	If a completely Hausdorff space $X$ is not totally disconnected then there exists a hyper-real maximal ideal $M$ in $B_1(X)$.
\end{corollary}
\begin{proof}
	It is enough to prove that there exists an unbounded Baire one function in $B_1(X)$. We know that if a completely Hausdorff space is not totally disconnected, then there always exists an unbounded Baire one function \cite{AA}.
\end{proof}
\noindent In the next theorem we characterize the real maximal ideals of $B_1(X)$.
\begin{theorem}
	For the maximal ideal $M$ of $B_1(X)$ the following statements are equivalent:
	\begin{enumerate}
		\item $M$ is a real maximal ideal.
		\item $Z_B[M]$ is closed under countable intersection.
		\item $Z_B[M]$ has countable intersection property.
	\end{enumerate}
\end{theorem}
\begin{proof}
	$(1) \implies (2)$: Assume that $(2)$ is false, i.e. there exists a sequence of functions $\{f_n\}$ in $M$ for which $\bigcap\limits_{n=1}^\infty Z(f_n) \notin Z_B[M]$. Set $f(x)=\sum\limits_{n=1}^\infty \big(|f_n(x)| \wedge \frac{1}{4^n}\big)$, $\forall x\in X$. It is clear that, the function $f$ defined on $X$ is actually a Baire one function (\cite{AA}) and $Z(f)= \bigcap\limits_{n=1}^\infty Z(f_n)$. Thus, $Z(f) \notin Z_B[M]$. Hence $f \notin M \implies M(f) >0$ in $B_1(X)/M$.\\ Fix a natural number $m$. Then $Z(f_1) \bigcap Z(f_2)\bigcap Z(f_3)...\bigcap Z(f_m)=Z$(say) $\in Z_B[M]$. Now for any point $x \in Z$, $f(x)= \sum\limits_{n=m+1}^\infty\big(|f_n(x)| \wedge \frac{1}{4^n}\big) \leq \sum\limits_{n=m+1}^\infty \frac{1}{4^n}= 3^{-1} 4^{-m}$. This shows that, $0 < M(f) \leq M(3^{-1}4^{-m})$, $\forall$ $m \in \mathbb{N}$. Hence $M(f)$ is an infinitely small member in $B_1(X)/M$. So, $M$ becomes a hyper-real maximal ideal and then $(1)$ is false.\\
	$(2) \implies (3)$: Trivial, as $\emptyset \notin Z_B[M]$.\\
	$(3) \implies (1)$: Assume that $(1)$ is false, i.e. $M$ is hyper-real. So, there exists $f \in B_1(X)$ so that $|M(f)|$ is infinitely large in $B_1(X)/M$. Therefore for each $n \in \mathbb{N}$, $Z_n$ defined in Theorem \ref{3eqv}, belongs to $Z_B[M]$. Since $\mathbb{R}$ is archimedean, we have $\bigcap \limits_{n=1}^\infty Z_n = \emptyset$. Thus $(3)$ is false.
\end{proof}
\noindent So far we have seen that for any topological space $X$, all fixed maximal ideals of $B_1(X)$ are real. Though the converse is not assured in general, we show in the next example that in $B_1(\mathbb{R})$ a maximal ideal is real if and only if it is fixed.
\begin{example}
	Suppose $M$ is any real maximal ideal in $B_1(\mathbb{R})$. We claim that $M$ is fixed. The identity $i: \mathbb{R} \rightarrow \mathbb{R}$ belongs to $B_1(\mathbb{R})$. Since $M$ is a real maximal ideal, there exists a real number $r$ such that $M(i)=M(\mathbf{r})$. This implies $i - \mathbf{r} \in M$. Hence $Z(i-\mathbf{r})\in Z_B[M]$. But $Z(i-\mathbf{r})$ is a singleton. So, $Z_B[M]$ is fixed, i.e., $M$ is fixed.
\end{example} 
\noindent In view of Observation~\ref{Obs}(3), we conclude that a maximal ideal $M$ in $B_1(\mathbb{R})$ is real if and only if there exists a unique $p \in \mathbb{R}$ such that $\chi_p - 1 \in M$.\\\\
\noindent If $X$ is a P-space then $C(X)$ possesses real free maximal ideals. In such case however, $B_1(X) = C(X)$.  Consequently, $B_1(X)$ possesses real free maximal ideals, when $X$ is a P-space. It is still a natural question, what are the topological spaces $X$ for which $B_1(X)$ ($\supseteq C(X)$) contains a free real maximal ideal?

\end{document}